\documentclass[12pt]{article}
\usepackage{amsmath,amssymb}

\newcommand{\Var}{{\cal{V}_{\mathbb{C}}}}

\def\1{\underline{1}}

\def\AA{{\mathbb A}}

\def\LLL{{\mathbb L}}
\def\Z{{\mathbb Z}}

\def\C{{\mathbb C}}
\def\H{{\mathbb H}}

\def\A{{\mathcal A}}

\def\CP{\mathbb C\mathbb P}

\newtheorem{theorem}{Theorem}

\newenvironment{corollary}
{\smallskip\noindent{\bf Corollary\/}.}{\smallskip\par}
\newenvironment{example}
{\smallskip\noindent{\bf Example\/}.}{\smallskip\par}

\newenvironment{remarks}
{\smallskip\noindent{\bf Remarks\/}.}{\smallskip\par}
\newenvironment{proof}
{\noindent{\bf Proof\/}.}{{ $\square$}\smallskip\par}

\newenvironment{conjecture}
{\smallskip\noindent{\bf Conjecture\/}:}{\smallskip\par}

\title{On generating series of classes of equivariant Hilbert schemes of fat points
}
\author{S.M.~Gusein-Zade \thanks{Partially supported by the grants
RFBR-007-00593, NSh-709.2008.1.
Address: Moscow State University, Faculty
of Mathematics and Mechanics, Moscow, GSP-1, 119991, Russia. E-mail:
sabir\symbol{'100}mccme.ru} \and I.~Luengo \and
A.~Melle--Hern\'andez \thanks{The last two authors were partially
supported by the grant MTM2007-67908-C02-02. Address: University
Complutense de Madrid, Dept. of Algebra, Madrid, 28040, Spain.
E-mail: iluengo\symbol{'100}mat.ucm.es,
amelle\symbol{'100}mat.ucm.es}}

\date{}
\begin{document}
\def\eps{\varepsilon}

\maketitle
\begin{abstract}
In previous papers the authors gave formulae for generating series of classes (in the Grothendieck ring $K_0(\Var)$ of complex quasi-projective varieties) of Hilbert schemes of zero-dimensional subschemes on smooth varieties and on orbifolds in terms of certain local data and the, so called, power structure over the ring $K_0(\Var)$. Here we give an analogue of these formulae for equivariant (with respect to an action of a finite group on a smooth variety) Hilbert schemes of zero-dimensional subschemes and compute some local generating series for an action of the cyclic group on a smooth surface.
\end{abstract}

%%%%%%%%%%%%%%%%%%%%%%%%%%%%%%%%%%%%%%%%%%%%%%%%%%%%%%%%%%%%%%%%%%%%%%%%%%%%%%%
\section{Introduction}
For a complex $d$-dimensional quasi-projective variety $X^d$,
let $\mbox{Hilb}^k_X$ be the Hilbert scheme of zero-dimensional subschemes (sets of ``fat points'') of length $k$ of $X$. For a locally closed subvariety $Y\subset X$, let us denote by $\mbox{Hilb}^k_{X,Y}$ the Hilbert scheme of zero-dimensional subschemes of length $k$ of the variety $X$ supported at points of the variety $Y$, and for a point $x\in X$, $\mbox{Hilb}^k_{X,x}:=\mbox{Hilb}^k_{X,\{x\}}$.

\medskip
Let $K_0(\Var)$ be the Grothendieck ring of complex quasi-projective
varieties. This is the abelian group generated by the classes $[X]$ of all complex quasi-projective varieties $X$ modulo the relations:
\begin{enumerate}
\item[1)] if varieties $X$ and $Y$ are isomorphic, then $[X]=[Y]$;
\item[2)] if $Y$ is a Zariski closed subvariety of $X$, then $[X]=[Y]+[X\setminus Y]$.
\end{enumerate}
The multiplication in $K_0(\Var)$ is defined by the Cartesian product
of varieties: $[X_1]\cdot [X_2]=[X_1\times X_2]$.
The class $[\AA^1_\C]\in K_0(\Var)$ of the complex affine line is
denoted by $\LLL$. For a quasi-projective variety $X$, the class $[X]$,
being an additive invariant of the variety, can be considered as a \emph{generalized Euler characteristic} $\chi_g(X)$
of the variety $X$. The additivity of $\chi_g(\bullet)$ (property 2) above) permits to use it as a measure for the notion of the integral with respect to the generalized Euler characteristic. If $\psi:X\to \A$ is a \emph{constructible function} on a variety $X$ with values in an abelian goup $\A$, then the \emph{integral of $\psi$ with respect to the generalized Euler characateristic} is defined  as
\begin{equation}\label{Euler_int}
\int_X \psi  {d\chi_g}:=\sum_{a\in \A} \chi_g(\psi^{-1}\{a\})\cdot a \in K_0(\Var)\otimes_\Z \A.
\end{equation}

\medskip
Let
$$
\H_X(T):=1+\sum\limits_{k=1}^\infty\,[\mbox{Hilb}^k_X]\,T^k \in 1+T \cdot K_0(\Var)[[T]]\quad \mbox{and}
$$
$$\H_{X,Y}(T):= 1+\sum\limits_{k=1}^\infty\,[\mbox{Hilb}^k_{X,Y}]\,T^k\in 1+T\cdot K_0(\Var)[[T]]\,\qquad$$
be the generating series of classes of Hilbert schemes.

\medskip
In \cite{MRL}, there was defined a notion of a \emph{power structure} over a ring and there was described a natural power structure over
the Grothendieck ring $K_0(\Var)$ of complex quasi-projective
varieties. This means that for a series $A(T)=1+a_1T+a_2T^2+\ldots \in 1+T \cdot K_0(\Var)[[T]]$ and for an element $m\in K_0(\Var)$
one defines a series $(A(T))^m\in 1+ T \cdot K_0(\Var)[[T]]$ so that all the usual properties of the exponential function hold.
For the natural power structure over the ring $K_0(\Var)$ and for $a_i=[A_i], m=[M]$ ,  where $A_i$ and $M$ are quasi-projective varieties, the series $(A(T))^m$ has the following  geometric description.
%% given
The coefficient at $T^k$
in the series
$$
(1+[A_1]T+[A_2]T^2+\ldots)^{[M]}
$$
is represented by the configuration space of pairs $(K,\varphi)$ consisting of a finite subset $K$ of the variety $M$ and a map $\varphi$ from $K$ to the disjoint
union $\coprod_{i=1}^\infty A_i$ of the varieties $A_i$, such that $\sum_{x\in K}I(\varphi(x))=k$, where $I:\coprod_{i=1}^\infty A_i\to {\Z}$ is the tautological function sending the component $A_i$ of the disjoint union to $i$.
This power structure is connected with the $\lambda$-structure on the ring $K_0(\Var)$ defined by the Kapranov zeta function \cite{Kapranov}:
$$
\zeta_{M}(T)=1+[S^1 M]\cdot T+[S^2 M]\cdot t^2+[S^3X]\cdot T^3+\ldots 
$$
where $S^k M$ is the $k$-th-symmetric power of the variety $M$: $\zeta_{M}(T)=(1-T)^{-[M]}.$

\medskip
There are two natural homomorphisms from the Grothendieck ring $K_0(\Var)$ to the ring $\Z$ of integers and to the ring $\Z[u,v]$ of polynomials in two variables: the Euler characteristic (alternating sum of ranks of cohomology groups with
compact support) $\chi:K_0(\Var)\to \Z$ and the Hodge-Deligne polynomial
$e:K_0(\Var)\to \Z[u,v]$. These homomorphisms respect the power
structures over the corresponding rings (see e.g. \cite{Michigan}).

\medskip
In \cite{Michigan}, it was shown that, for a smooth quasi-projective variety $X$ of dimension $d$, the following equation holds:
\begin{equation}\label{eq1}
\H_X(T)=\left(\H_{\AA^d_{\C},0}(T)\right)^{[X]}
\end{equation}
where $\AA^d_{\C}$ is the complex affine space of
dimension $d$. For $d=2$, i.e. for surfaces, in other terms this
equation was proved in the Grothendieck ring of motives by
L.~G\"ottsche~\cite{Got2}. In this case one has
\begin{equation}\label{got}
\H_{\AA^2_{\C},0}(T)=\prod_{i=1}^{\infty}\frac{1}{1-\LLL^{i-1} T^i}\,.
\end{equation}
For an arbitrary dimension $d$, the
reduction of the equation (\ref{eq1}) for the Hodge-Deligne
polynomial
%% homomorphism $e:K_0(\Var)\to \Z[u,v]$
was proved by J.~Cheah in \cite{Cheah}.

\medskip
A generalization of the equation (\ref{eq1}) for orbifolds was given in \cite{orbifolds}.
In this case the function $\H_{X,x}(T)$ 
is a constructible function on $X$ with values in the abelian group $1+T\cdot K_0(\Var)[[T]]$ (with respect to multiplication)
and one has 
\begin{equation}\label{eq2}
\H_X(T)=\int_X \H_{X,x}(T)^{d\chi_g}.
\end{equation}
(Here $d\chi_g$ is put into the exponent since the group operation in $1+T \cdot K_0(\Var)[[T]]$ is the multiplication. In this case the sum in RHS of the equation (\ref{Euler_int}) is substituted by the product.)
The equation (\ref{eq2}) reduces the computation of the generating series $\H_X(T)$ to the computation of the local data $\H_{X,x}(T)$.

\medskip
Here we shall give an analogue of equation (\ref{eq2}) for equivariant (with respect to an action of a finite group $G$ on a smooth variety) Hilbert schemes of zero dimensional subschemes and compute some local generating series for an action of a cyclic group on a smooth surface.

%%%%%%%%%%%%%%%%%%%%%%%%%%%%%%%%%%%%%%%%%%%%%%%%%%%%%%%%%%%%%%%%%%%%%%%%%%%%%%%%
\section{Equivariant Hilbert scheme of fat points}
Let $G$ be a finite group (of order $|G|$) acting on a smooth complex $d$-dimensional quasi-projective complex variety $X^d$. The group $G$ also acts on the Hilbert schemes $\mbox{Hilb}^k_X$ of zero-dimensional subschemes on $X$. One can say that there are (at least) three natural notions of \emph{equivariant Hilbert schemes} of zero-dimensional subschemes on $X$ (see e.g. \cite{cmt}, \cite{kuznetsov}, \cite{nolla}).

First, one can define the equivariant Hilbert scheme ${}^{(1)}\mbox{Hilb}^{G,k}_{X}$ as the $G$-invariant part of the action of the group $G$ on $\mbox{Hilb}^{k}_{X}$.

Second, as the equivariant Hilbert scheme ${}^{(2)}\mbox{Hilb}^{G,k}_{X}$ one can take the (unique) component (or union of components if the variety $X$ is reducible) of ${}^{(1)}\mbox{Hilb}^{G,k}_{X}$ which maps birationally on the $k/|G|$-th symmetric power of $X$. 

A $G$-invariant zero-dimensional subscheme of lengh $k$ (i.e. a closed pont of ${}^{(1)}\mbox{Hilb}^{G,k}_{X}$) has 
a decomposition into parts supported at different $G$-orbits. For each of these parts one has the natural representation of the group $G$ on the fibre of the tautological bundle on the Hilbert scheme.
The third version  ${}^{(3)}\mbox{Hilb}^{G,k}_{X}$ of the equivariant Hilbert scheme consists of those points of  ${}^{(1)}\mbox{Hilb}^{G,k}_{X}$ for which all these representations corresponding to parts supported at different $G$-orbits are  multiples of the regular one. 

In the last two cases $k$ must be a multiple of the order $|G|$ of the group $G$. One always has 
$$
{}^{(1)}\mbox{Hilb}^{G,k}_{X}\supset {}^{(3)}\mbox{Hilb}^{G,k}_{X}\supset {}^{(2)}\mbox{Hilb}^{G,k}_{X}
$$
and all these equivariant Hilbert schemes are quasi-projective varieties. They are smooth if $\mbox{Hilb}^{k}_{X}$ is smooth. In particular this holds if $X$ is a smooth surface ($d=2$). In many cases, in particular for surfaces, the last two notions coincide. This is not true in general (see e.g. \cite{cmt}).
The equivariant Hilbert schemme  ${}^{(1)}\mbox{Hilb}^{G,k}_{X}$ is always larger than the other two. In particular 
${}^{(1)}\mbox{Hilb}^{G,1}_{\C^d}\ne \emptyset$. It seems that the last two notions are more interesting from  geometrical point of view. In particular, for a finite group $G \subset SL(2,\C)$ acting on $\C^2$ in the natural way, ${}^{(2)}\mbox{Hilb}^{G,|G|}_{\C^2}(={}^{(3)}\mbox{Hilb}^{G,|G|}_{\C^2})$ is a crepant resolution of the factor space $\C^2/G$. However, the first one could be interesting as well. In particular, it seems that formulae for the generating series of classes of 
${}^{(\bullet)}\mbox{Hilb}^{G,k}_{\C^2}$ (or of ${}^{(\bullet)}\mbox{Hilb}^{G,k}_{\C^2,0}$) are somewhat simpler in this case. 

Let 
$$
{}^{(\bullet)}\H_X^G(T):=1+\sum\limits_{k=1}^\infty\,[{}^{(\bullet)}\mbox{Hilb}^{G,k}_X]\,T^k 
$$
and, for a locally closed $G$-invariant subvariety $Y\subset X$, 
$${}^{(\bullet)}\H^G_{X,Y}(T):= 1+\sum\limits_{k=1}^\infty\,[{}^{(\bullet)}\mbox{Hilb}^{G,k}_{X,Y}]\,T^k$$
be the generating series of classes of the equivariant Hilbert schemes.

%%%%%%%%%%%%%%%%%%%%%%%%%%%%%%%%%%%%%%%%%%%%%%%%%%%%%%%%%%%%%%%%%%%%%%%%%%%%%
\section{Generating series of classes of equivariant Hilbert scheme of fat points through local data}

Statements of this section have the same form for all three notions of the equivariant Hilbert scheme of zero-dimensional subschemes discussed above. 
Since we do not have to specify the notion we shall use the notation ${}^{(\bullet)}\mbox{Hilb}^{G,k}_{X}$, ${}^{(\bullet)}\mbox{Hilb}^{G,k}_{X,x}, \ldots$ without indicating the number corresponding to the version.

\medskip
Let $G$ be a finite group (of order $|G|$) acting on a smooth $d$-dimensional quasi-projective complex variety $X^d$, let $Y:=X^d/G$ be the corresponding factor space and let $p:X\to Y$ be the projection map.
For a point $x\in X$, let $G_{x}=\{g\in G: g*x= x \}\subset G$ be the isotropy subgroup of the point $x$.
Let $y$ be a point of $Y=X^d/G$, i.e. a $G$-orbit in $X$. For all points $x$ in this $G$-orbit the isotropy subgroups $G_x$ are conjugate to each other. One has the natural representation $\alpha_x$ of the isotropy group $G_x$ on the tangent space $T_xX \cong \C^d$ at the point $x$.

For a conjugate class $h$ of subgroups of $G$, for a representative $H\subset G$ of this class, and for a representation $\alpha$ of the group $H$ on $\C^d$, let $Y_{h,\alpha}$ be the set of points $y \in Y$ such that the isotropy subgroup of each point in the corresponding orbit belongs to $h$ and (for those of them whose isotropy group coincides with $H$) the representation of the isotropy subgroup coincides with $\alpha$. For a representation $\alpha$ of a subgroup $H\subset G$ on $\C^d$, one considers the equivariant Hilbert scheme  ${}^{(\bullet)}\mbox{Hilb}^{H,k}_{\C^d,0}$
and the corresponding generating series 
$${}^{(\bullet)}\H^H_{\C^d,\alpha}(T):=1+\sum\limits_{k=1}^\infty\,[{}^{(\bullet)}\mbox{Hilb}^{H,k}_{\C^d,0}]\,T^k \, .
$$

\begin{theorem}\label{theo1}
The following equations holds:
\begin{equation}\label{eqtheo1}
{}^{(\bullet)}\H_X^G(T)=\prod_{h,\alpha}\left( {}^{(\bullet)}\H^H_{\C^d,\alpha}(T^{|G|/|H]})\right)^{[Y_{h,\alpha}]}= \int_{X/G} {}^{(\bullet)}\H_{X,x}^{G_x}(T^{|G|/|G_x|})^{d\chi_g}.
\end{equation}
\end{theorem}

\begin{proof} The proof essentially follows the lines of the proof of Theorem 1 from \cite{orbifolds}. If $Z$ is a $G$-invariant Zariski closed subset
of the variety $X$, one has:
\begin{equation}\label{eq5}
{}^{(\bullet)}\H_X^G(T)= {}^{(\bullet)}\H^G_{X,Z}(T) \cdot {}^{(\bullet)}\H^G_{X,X\setminus Z}(T).
\end{equation}
This implies that one has to prove  that ${}^{(\bullet)}\H^G_{X, p^{-1}(Y_{h,\alpha})}(T)=\left( {}^{(\bullet)}\H^H_{\C^d,\alpha}(T^{|G|/|H]})\right)^{[Y_{h,\alpha}]}$. 
Moreover, by the same reason, it is sufficient to prove that ${}^{(\bullet)}\H^G_{X, p^{-1}(Y_i)}(T)=\left( {}^{(\bullet)}\H^H_{\C^d,\alpha}(T^{|G|/|H]})\right)^{[Y_{i}]}$ for elements $Y_i$ of a Zariski open covering of $Y_{h,\alpha}.$ Without loss of generality one may assume that $X$ (and therefore $p^{-1}(Y_{h,\alpha})$) lies in an affine space $\AA_{\C}^N$ (an affine chart of the ambient projective space).

Let us fix linear equations corresponding to the representation $\alpha$ of the group $H$:
\begin{equation}\label{repr}
g^*z_i=\sum\limits_{j=1}^{d}\alpha_{i,j}(g)z_j\,,
\end{equation}
where $(\alpha_{i,j}(g)) = \alpha(g)$.
For a point $x\in p^{-1}(Y_{h,\alpha})$ with the isotropy subgroup $G_x$ coinciding with $H$ (not only conjugate to it), let $u_1$, $u_2$, \dots, $u_d$ be a regular system of parameters on the manifold $X$ at the point $x$. For example, one may suppose that $x$ is the origin in $\AA_{\C}^N$ and $u_1$, $u_2$, \dots, $u_d$ are standard coordinates in $\AA_{\C}^N$ such that the projection of the tangent space $T_xX$ to the corresponding $d$-dimensional coordinate subspace is non-degenerate. In this case $u_1-u_1(x')$, $u_2-u_2(x')$, \dots, $u_d-u_d(x')$ is a regular system of parameters at each point $x'$ from a Zariski open neighbourhood of the point $x$ in $X$. Moreover, let us suppose that the parameters $u_1$, $u_2$, \dots, $u_d$ are choosen in such a way that the representation of the subgroup $H$ in the tangent space $T_xX$ is given by the standard equations (\ref{repr}). Define functions $\widetilde{u}_1$, $\widetilde{u}_2$, \dots, $\widetilde{u}_d$ on $X$ by the equations
$$
\widetilde{u}_i=\frac{1}{\vert H\vert}\sum\limits_{g\in H}\sum\limits_{j=1}^{d}\alpha_{i,j}(g^{-1})g^*u_j\,.
$$
One has
$$
g^*\widetilde{u}_i=\sum\limits_{j=1}^{d}\alpha_{i,j}(g)\widetilde{u}_j\,.
$$
At the point $x$ one has $d\widetilde{u}_i=du_i$. Therefore $\widetilde{u}_1-\widetilde{u}_1(x')$, $\widetilde{u}_2-\widetilde{u}_2(x')$, \dots, $\widetilde{u}_d-\widetilde{u}_d(x')$ is a regular system of parameters at each point from a Zariski open neighbourhood $Z_x$ of the point $x$ in the corresponding irreducible component of $p^{-1}(Y_{h,\alpha})$. The choice of the regular set of parameters identifies $H$-invariant zero-dimensional subschemes on $X$ supported at a point $x'$ from $Z_x$ and thus $G$-invariant zero-dimensional subschemes on $X$ supported at $GZ_x$ with $H$-invariant zero-dimensional subschemes on the space $\AA^d_{\C}$ (with the action defined by the equations (\ref{repr})) supported at the origin.

This way, a $G$-invariant zero-dimensional subscheme on $X$
supported at points of the subvariety $GZ_x$ is defined by a finite
subset $K\subset Y_x=p(Z_x)\equiv Z_x$ to each point of which there corresponds an $H$-invariant
zero-dimensional subscheme on $\AA^d_{\C}$ supported at the
origin. The length of the subscheme is equal to $\vert G\vert/\vert H\vert$ times the sum of lengths
of the corresponding subschemes of $\AA^d_{\C}$. As it
follows from the geometric description of the power structure over
the Grothendieck ring of quasi-projective varieties, the
coefficient at $T^n$ in the right hand side of the equation
(\ref{eqtheo1}) is represented just by the configuration space of such
objects. This proves the statement.
\end{proof}

%%%%%%%%%%%%%%%%%%%%%%%%%%%%%%%%%%%%%%%%%%%%%%%%%%%%%%%%%%%%%%%%%%%%%%%%%%%%%
\section{Generating series of classes of equivariant Hilbert schemes for two-dimensional representations of a cyclic group.}
Let the cyclic group $\Z_M$ act on the plane $\C^2$ by $\sigma*(x,y)=(\sigma x, \sigma^N y)$ where $\sigma=\exp\left(\frac{2\pi i}{M}\right)$ is the generator of $\Z_M$. For $N=-1$ (or rather $N\equiv -1 \mod M$) the factor space $\C^2/\Z_M$ has the $A_{M-1}$ singularity.

We shall denote ${}^{(\bullet)}\H^{\Z_M}_{\C^2}(T)$, ${}^{(\bullet)}\H^{\Z_M}_{\C^2,0}(T)$, {\dots} for this action by ${}^{(\bullet)}\H^{M, N}_{\C^2}(T)$, ${}^{(\bullet)}\H^{M,N}_{\C^2,0}(T)$, \dots

\begin{theorem} The following equation holds:
\begin{eqnarray}\label{eqth2}
{}^{(1)}\H^{M,-1}_{\C^2,0}(T)&=&\prod^{\infty}_{i=1}\left(\frac{(1-T^{Mi})^M}{1-T^i}\cdot\frac{1}{(1-\LLL^iT^{Mi})^{M-1}\cdot (1-\LLL^{i-1}T^{Mi})}{}\right)\,,\notag \\
{}^{(2)}\H^{M,-1}_{\C^2,0}(T)&=&\prod^{\infty}_{i=1}\frac{1}{(1-\LLL^iT^{Mi})^{M-1}\cdots(1-\LLL^{i-1}T^{Mi})}{}\,.
\end{eqnarray}
\end{theorem}

\begin{proof}
It appears to be somewhat simpler to describe the computation of ${}^{(\bullet)}\H^{M,-1}_{\C^2,\C}(T)$ where $\C$ is an invariant line of the $\Z_M$-action on the plane $\C^2$ and then to apply Theorem \ref{theo1} to get ${}^{(\bullet)}\H^{M,-1}_{\C^2,0}(T)$. To compute ${}^{(\bullet)}\H^{M,N}_{\C^2}(T)$, ${}^{(\bullet)}\H^{M,N}_{\C^2,\C}(T)$, or ${}^{(\bullet)}\H^{M,N}_{\C^2,0}(T)$, one uses the method of G. Ellingsrud and S.A. Str\o mme \cite{ell-str} based on a result of A. Bialynicki-Birula. For that, one considers the natural action of the complex 2-torus $\C^*\times\C^*$ on the projective plane $\CP^2$ and the corresponding action on the Hilbert schemes of zero-dimensional subschemes on it. This action has a finite number of fixed points. For a subgroup $G\subset \C^*\times\C^*$ (say, for a cyclic one) the torus acts on the corresponding equivariant Hilbert schemes as well. This action (with a fixed generic subgroup of $\C^*\times\C^*$ isomorphic to $\C^*$) defines cell decompositions of the Hilbert schemes of zero-dimensional suschemes and of the equivariant one(s). Cells (locally closed subvarieties isomorphic to complex affine spaces) correspond to fixed points of the action on the Hilbert schemes. The dimension of a cell is equal to the dimension of the subspace of the tangent space corresponding to representations of $\C^*$ with positive characters. Fixed points of the natural action of the torus $\C^*\times\C^*$ on $\mbox{Hilb}_{\C^2}^k$ and also on ${}^{(1)}\mbox{Hilb}_{\C^2}^{M,N;k}$ are the monomial ideals in $\C[[x,y]]$ of length (codimension) $k$. Monomial ideals of length $k$ correspond to partitions of $k$, i.e. to Young diagrams of size $k$. Fixed points of the ($\C^*\times\C^*$)-action on ${}^{(2)}\mbox{Hilb}_{\C^2}^{M,N;k}$ are those monomial ideals of length $k$, for which the corresponding Young diagram has the same numbers of boxes with the quasi-homogeneous weights $0$, $1$, {\dots}, $M-1$. (The box corresponding to the monomial $x^ky^{\ell}$ has the quasi-homogeneous weight equal to $k+N\ell \mod{M}$. In the last case the Hilbert scheme (and the set of fixed points) is empty if $k$ is not divisible by $M$.)

From \cite{ell-str}, it follows that the dimension of the cell in ${}^{(\bullet)}\mbox{Hilb}^{M,-1; k}_{\C^2,\C}$ corresponding to the fixed point described by a Young diagram of the partition $\{b_0\ge b_1\ge \ldots\ge b_{r-1}>0\}$ of the integer $k$ is equal to the number of monomials in the (equivariant) expression
$$
T^+=\sum_{1\le i\le j\le r}\sum_{s=b_j}^{b_{j-1}-1} \lambda^{i-j-1}\mu^{b_{i-1}-s-1}\, ,
$$
for the ``positive part" of the tangent space to the Hilbert scheme $\mbox{Hilb}^{M,N;k}_{\C^2, \C}$ with the weights $N(i-j-1)+(b_{i-1}-s-1)\equiv 0 \mod M$. For $N=-1$ these weights are just the lengths of hooks of the corresponding Young diagram. For a fixed $M$, Young diagrams of size $k$ are in one-to-one correspondence with the sets consisting of the so-called $M$-core (of size $k'$) of the diagram (empty for diagrams with equal numbers of boxes with different weights) and its ``star $M$-diagram", i.e. a collection of ``usual" diagrams of sizes, say, $k_0$, $k_1$, \dots,  $k_{M-1}$, such that $k'+M(k_0+k_1+\ldots+k_{M-1})=k$: see, e.g., \cite{repres} The dimension of the corresponding cell (i.e. the number of hooks with length divisible by $M$) is equal to $k_0+k_1+\ldots+k_{M-1}$. For ${}^{(2)}\mbox{Hilb}^{M,-1; k}_{\C^2,\C}$ ($k'=0$) this means that the cells correspond to partition of sizes $k_0$, $k_1$, \dots,  $k_{M-1}$ for representations of $k/M$ as an ordered sum $k_0+k_1+\ldots+k_{M-1}$. The dimension of the corresponding cell is equal to $k_0+k_1+\ldots+k_{M-1}$. Therefore
$$
{}^{(2)}\H^{M,-1}_{\C^2,\C}(T)=\prod^{\infty}_{i=1}\frac{1}{(1-\LLL^iT^{Mi})^{M}}\,.
$$
The generating series for the numbers of $M$-cores of different sizes is
$$
\prod^{\infty}_{i=1}\frac{(1-T^{Mi})^M}{1-T^i}
$$
(see, e.g., \cite{chen}). Therefore
$$
{}^{(1)}\H^{M,-1}_{\C^2,\C}(T)=\prod^{\infty}_{i=1}\frac{(1-T^{Mi})^M}{1-T^i}\cdot\prod_{i=1}^{\infty}\frac{1}{(1-\LLL^iT^{Mi})^{M}}
\,.
$$
From Theorem \ref{theo1}, equation (2) and the equation $(1-\LLL^i T)^{-\LLL}=(1-\LLL^{i+1}T)^{-1}$ it follows that
$$
{}^{(\bullet)}\H^{M,-1}_{\C^2,0}(T)={}^{(\bullet)}\H^{M,-1}_{\C^2,\C}(T)\cdot\left(\prod_{i=1}^{\infty}\frac{1}{1-\LLL^{i-1}T^{Mi}}\right)^{1-\LLL}
={}^{(\bullet)}\H^{M,-1}_{\C^2,\C}(T)\cdot\prod_{i=1}^{\infty}\frac{1-\LLL^{i}T^{Mi}}{1-\LLL^{i-1}T^{Mi}}\,.
$$
Therefore
$$
{}^{(1)}\H^{M,-1}_{\C^2,0}(T)=\prod^{\infty}_{i=1}\frac{(1-t^{Mi})^M}{1-T^i}\cdot
\prod^{\infty}_{i=1}\frac{1}{(1-\LLL^iT^{Mi})^{M-1}(1-\LLL^{i-1}T^{Mi})}
\,,
$$
$$
{}^{(2)}\H^{M,-1}_{\C^2,0}(T)=
\prod^{\infty}_{i=1}\frac{1}{(1-\LLL^iT^{Mi})^{M-1}(1-\LLL^{i-1}T^{Mi})}
\,.
$$
\end{proof}

\begin{corollary}
\emph{Let the cyclic group $\Z_M$ act on a smooth surface $S$ in such a way, that the factor space $S/\Z_M$ has only $A_{M-1}$ singularities $($i.e. at each of $d$ fixed points $P_1$, \dots, $P_d$ one has the representation corresponding to $N=-1$$)$. Then}
$$
{}^{(1)}\H^{M,-1}_{S}(T)=
\left(\prod^{\infty}_{i=1}\frac{(1-t^{Mi})^M}{(1-T^i)(1-\LLL^iT^{Mi})^{M-1}(1-\LLL^{i-1}T^{Mi})}\right)^{d}
\cdot\left(\prod^{\infty}_{i=1}\frac{1}{1-\LLL^{i-1}T^{Mi}}\right)^{[(S\setminus\{P_i\})/Z_M]}
\,.
$$
$$
{}^{(2)}\H^{M,-1}_{S}(T)=
\left(\prod^{\infty}_{i=1}\frac{1}{(1-\LLL^iT^{Mi})^{M-1}(1-\LLL^{i-1}T^{Mi})}\right)^{d}
\cdot\left(\prod^{\infty}_{i=1}\frac{1}{1-\LLL^{i-1}T^{Mi}}\right)^{[(S\setminus\{P_i\})/Z_M]}
\,.
$$
\end{corollary}

\begin{example}
Let the group $\Z_3$ act on the projective plane $\CP^2$ by $\sigma*(x_0:x_1:x_2)=(x_0:\sigma x_1:\sigma^2 x_2)$. Then
\begin{eqnarray*}
{}^{(2)}\H^{\Z_3}_{\CP^2}(T)&=&1+(1+7\LLL+\LLL^2)T^3+(1+8\LLL+36\LLL^2+8\LLL^3+\LLL^4)T^6 \\
&&+(1+8\LLL+44\LLL^2+149\LLL^3+44\LLL^4+8\LLL^5+\LLL^6)T^9 \\
&&+(1+8\LLL+45\LLL^2+192\LLL^3+543\LLL^4+192\LLL^5+45\LLL^6+8\LLL^7+\LLL^8)T^{12}+\ldots
\end{eqnarray*}
\end{example}

\begin{remarks}
{\bf 1.} One can easily see that if $N_1N_2\equiv 1 \mod M$, one has ${}^{(\bullet)}\H^{M,N_1}_{\C^2,0}(T)={}^{(\bullet)}\H^{M,N_2}_{\C^2,0}(T)$.

\noindent{\bf 2.} It seems that, for $N\ne-1$, one has somewhat better (less complicated) formulae for the series ${}^{(1)}\H^{M,N}_{\C^2,0}(t)$ than for the series ${}^{(2)}\H^{M,N}_{\C^2,0}(T)$ (at least in the form similar to (\ref{got}) and (\ref{eqth2}). To show that, it is convenient to write down the logarithms $\mbox{Log\,}{}^{(\bullet)}\H^{M,N}_{\C^2,0}(T)$ of the generating series ${}^{(\bullet)}\H^{M,N}_{\C^2,0}(T)$ in the sense of \cite{MRL}: if $A(T)=\prod\limits_{i,j}(1-\LLL^jT^i)^{-k_{ij}}$, with $k_{ij}\in\Z$, then by definition $\mbox{Log\,}A(T)=\sum\limits_{i,j}k_{ij}\LLL^jT^i$. In particular, the equation (\ref{eqth2}) means that 
$$
\mbox{Log\,}{}^{(2)}\H^{M,-1}_{\C^2,0}(T)=\sum_{i=1}^\infty \left((M-1)\LLL^i+\LLL^{i-1}\right) T^{M i}\,.
$$
Computations made by the use of Maple gave:
\begin{eqnarray*}
\mbox{Log\,}{}^{(1)}\H^{3,1}_{\C^2,0}(T)&=&T+\LLL T^2+T^3+\LLL T^4+\LLL^2T^5+\LLL T^6+\LLL^2T^7+\LLL^3T^8+\LLL^2T^9  \\
&&+\LLL^3T^{10}+\LLL^4T^{11}+\LLL^3T^{12}+\LLL^4T^{13}+\LLL^5T^{14}+\LLL^4T^{15}+\LLL^5T^{16} \\
&&+\LLL^6T^{17}+\LLL^5T^{18}+\LLL^6T^{19}+\LLL^7T^{20}+\LLL^6T^{21}+\ldots
\end{eqnarray*}
\begin{eqnarray*}
\mbox{Log\,}{}^{(2)}\H^{3,1}_{\C^2,0}(T)&=&(1+\LLL)T^3+(2\LLL+2\LLL^2+\LLL^3)T^6+(2\LLL^2+2\LLL^3+\LLL^4)T^9 \\ & &+(-\LLL^2+\LLL^3-\LLL^6)T^{12}  +(-\LLL^3-\LLL^5-\LLL^6-\LLL^7)T^{15}+(2\LLL^5+\LLL^7)T^{18} \\
&&+(2\LLL^4+3\LLL^5+7\LLL^6+6\LLL^7+6\LLL^8+3\LLL^9+2\LLL^{10})T^{21}+\ldots
\end{eqnarray*}

One can make the following

\begin{conjecture}
$$
{}^{(1)}\H^{3,1}_{\C^2,0}(T)=\prod_{i=1}^\infty \frac{1}{(1-\LLL^{i-1}T^{3i-2})(1-\LLL^{i}T^{3i-1})(1-\LLL^{i-1}T^{3i})}\,.
$$
\end{conjecture}

A conjectural equation for ${}^{(2)}\H^{3,1}_{\C^2,0}(T)$ is not clear.

Some other examples:
\begin{eqnarray*}
\mbox{Log\,}{}^{(1)}\H^{4,1}_{\C^2,0}(T)&=&T+\LLL T^2+T^3+\LLL T^4+T^5+(-1+\LLL+\LLL^2)T^6+T^7+(-1+\LLL+\LLL^2)T^8  \\
&& +T^9 +(-1+\LLL^2+\LLL^3)T^{10}+T^{11}+(-1+\LLL^2+\LLL^3)T^{12} \\
&&+T^{13}+(-1+\LLL^3+\LLL^4)T^{14}+T^{15} 
+(-1+\LLL^3+\LLL^4)T^{16}\\
&&+T^{17}+(-1+\LLL^4+\LLL^5)T^{18}+T^{19}+(-1+\LLL^4+\LLL^5)T^{20}+\ldots
\end{eqnarray*}
\begin{eqnarray*}
\mbox{Log\,}{}^{(2)}\H^{4,1}_{\C^2,0}(T)&=&(1+\LLL)T^4+(2\LLL+2\LLL^2+\LLL^3)T^8+(\LLL+4\LLL^2+5\LLL^3+3\LLL^4+\LLL^5)T^{12} \\
&& +(4\LLL^3+5\LLL^4+3\LLL^5)T^{16}+(-\LLL^2-3\LLL^3-2\LLL^4-\LLL^5-3\LLL^6-3\LLL^7 \\
&& -\LLL^8)T^{20}+\ldots 
\end{eqnarray*}
\begin{eqnarray*}
\mbox{Log\,}{}^{(1)}\H^{5,2}_{\C^2,0}(T)&=&T+T^2+\LLL T^3+\LLL T^4+T^5+\LLL T^6+\LLL T^7+\LLL^2T^8+\LLL^2T^9+\LLL T^{10} \\
&&+\LLL^2T^{11} 
+\LLL^2T^{12}+\LLL^3T^{13}+\LLL^3T^{14}+\LLL^2T^{15}+\LLL^3T^{16}+\LLL^3T^{17}+\LLL^4T^{18}\\ &&+\LLL^4T^{19}
+\LLL^3T^{20}+\LLL^4T^{21}+\LLL^4T^{22}+\LLL^5T^{23}+\LLL^5T^{24}+\LLL^4T^{25}+\ldots
\end{eqnarray*}
\begin{eqnarray*}
\mbox{Log\,}{}^{(2)}\H^{5,2}_{\C^2,0}(T)&=&(1+2\LLL)T^5+(3\LLL+5\LLL^2+2\LLL^3)T^{10}+(3\LLL^2+5\LLL^3+2\LLL^4)T^{15}\\ &&+(-3\LLL^2-2\LLL^3-4\LLL^4-3\LLL^5-3\LLL^6)T^{20}\\
&&+(-3\LLL^3-6\LLL^4-9\LLL^5-7\LLL^6-3\LLL^7)T^{25}+\ldots
\end{eqnarray*}

\noindent{\bf 3.} Though a conjectural formula for ${}^{(2)}\H^{M,1}_{\C^2,0}(T)$ (or for $\mbox{Log\,}{}^{(2)}\H^{M,1}_{\C^2,0}(T)$) is not clear even for small $M>2$, computations show that one could have the following stabilization. Let $\mbox{Log\,}{}^{(2)}\H^{M,1}_{\C^2,0}(T)=\sum\limits_{i=1}^{\infty}p_i^{M,1}(\LLL)\cdot T^{Mi}$, where $p_i^{M,1}(\LLL)$ are polynomials in $\LLL$. The computations predict that $p_i^{M',1}(\LLL)=p_i^{M'',1}(\LLL)$ for $M''>M'>i$.
\end{remarks}

%%%%%%%%%%%%%%%%%%%%%%%%%%%%%%%%%%%%%%%%%%%%%%%%%%%%%%%%%%%%%%%%%%%%%%%%%%%%%%%

\end{document}